\documentclass[a4paper,11pt]{article}

\usepackage[left=3.25cm, right=3.25cm, top=4cm]{geometry}

\usepackage{booktabs} 

\usepackage{mathptmx}       
\usepackage{helvet}         
\usepackage{courier}        
\usepackage{type1cm}  

\usepackage{amsmath}
\usepackage{amsthm}
\usepackage{amssymb}
\usepackage{amsfonts}
\usepackage{mathrsfs}
\usepackage{color}
\usepackage[utf8]{inputenc}
\usepackage[T1]{fontenc}
\usepackage{enumerate}
\usepackage[english]{babel}
\usepackage[bookmarks=false,colorlinks=true,linkcolor=blue,citecolor=red]{hyperref}
\usepackage[babel]{csquotes} 

\usepackage{graphicx}
\usepackage{latexsym}
\usepackage{epsfig}

\usepackage{thmtools,thm-restate}

\usepackage[ruled]{algorithm2e}

\SetAlFnt{\small}
\SetAlCapFnt{\small}
\SetAlCapNameFnt{\small}
\SetAlCapHSkip{0pt}
\IncMargin{-\parindent}

\theoremstyle{plain}
\newtheorem{theorem}{Theorem}[section]
\newtheorem{lemma}[theorem]{Lemma}
\newtheorem{proposition}[theorem]{Proposition}

\theoremstyle{definition}
\newtheorem{definition}[theorem]{Definition}

\theoremstyle{definition}

\makeindex             

\DeclareMathOperator{\Q}{\mathbb Q}
\DeclareMathOperator{\R}{\mathbb R}

\DeclareMathOperator{\vcsp}{VCSP}

\DeclareMathOperator{\dom}{dom}

\DeclareMathOperator{\ar}{ar}
\DeclareMathOperator{\QQ}{\mathbb Q\cup\{+\infty\}}
\DeclareMathOperator{\RR}{\mathbb R\cup\{+\infty\}}

\DeclareMathOperator{\val}{val}
\DeclareMathOperator{\attr}{attr}

\DeclareMathOperator{\obj}{obj}

\begin{document}

\title{Piecewise Linear Valued Constraint Satisfaction Problems with Fixed Number of Variables}  \date{}

\author{Manuel Bodirsky \thanks{The authors have received funding from the European Research Council (ERC) under the European	Union's Horizon 2020 research and innovation programme (grant agreement No 681988, CSP-Infinity).}\\{\small  Institut f\"{u}r Algebra}\\{\small Technische Universit\"{a}t  Dresden} \and Marcello Mamino \footnotemark[1]\\{\small Dipartimento di Matematica}\\ {\small Universit\`{a} di Pisa} \and Caterina Viola\footnotemark[1] \thanks{This author has been supported by DFG Graduiertenkolleg 1763 (QuantLA).}\\{\small Institut f\"{u}r Algebra}\\ {\small Technische Universit\"{a}t  Dresden}}
\maketitle
\abstract{Many combinatorial optimisation problems can be modelled as \emph{valued constraint satisfaction problems}. In this paper, we present a polynomial-time algorithm solving the valued constraint satisfaction problem for a fixed number of variables and for piecewise linear cost functions. Our algorithm finds  the infimum of a  piecewise linear function and decides whether it is a proper minimum.}

\section{Introduction}

 The input of a \emph{valued constraint satisfaction problem}, or \emph{VCSP} for short,  is a finite set of cost functions depending on a given finite set of variables, and the computational task is to find an assignment of values for the variables that minimises  the sum of the cost functions. Many computational optimisation problems arising in industry, business, manufacturing, and science, can be modelled as a VCSP. 

 VCSPs have been extensively studied in  the case of cost functions defined on a fixed finite set (the \emph{domain}).
The computational complexity of solving VCSPs depends on the set of allowed cost functions and has recently been classified if the domain is finite  \cite{KolmogorovThapperZivny,GenVCSP15,KozikOchremiak15,ZhukFVConjecture,BulatovFVConjecture}: every finite domain VCSP is either  polynomial-time tractable or it is NP-complete.

However, many outstanding combinatorial optimisation problems can be formulated as VCSPs only by allowing cost functions defined on infinite domains, e.g., the  set ${\mathbb Q}$ of rational numbers. 

Despite the interest in concrete VCSPs over the set  of rational numbers and over other infinite numeric domains (e.g., the integers, the reals, or the complex numbers), VCSPs over infinite domains have not yet been investigated systematically. 
The class of VCSPs for all sets of cost functions defined on arbitrary infinite domains is too large to allow  general complexity results. Indeed, every computational problem is polynomial-time Turing-equivalent to the VCSP for a suitable set of cost functions over an infinite domain~\cite{BodirskyGrohe}.
Therefore, we need to restrict the class of cost functions that we focus on.
One restriction that captures a great variety 
of theoretically and practical interesting optimisation problems is the class of all \emph{piecewise linear} (\emph{PL}) cost functions over ${\mathbb Q}$, i.e., $\Q$-valued\footnote{In the PL setting the domains $\Q$ and $\R$ are interchangeable; we only require the coefficients of the cost functions to be rational as we need to manipulate them computationally.} partial functions, whose graph is the union of linear half spaces.
In general, the VCSP for PL cost functions is NP-complete. Indeed,  the containment in NP follows from the fact that  the VCSP for the class of \emph{all} PL cost functions is equivalent to the existential theory of $(\Q;\leq,+,1)$, which is in NP (see \cite{Bodirsky2017ConstraintSP}). The NP-hardness follows from the fact that there exist NP-complete problems, e.g., the \textsc{Minimum Correlation Clustering}, and  the \textsc{Minimum Feedback Arc Set} problem (see \cite{Karp1972, GareyJohnson}), which can be formulated as instances of the VCSP for PL cost functions. 

In this paper, we prove that the restriction of the VCSP for \emph{all} PL cost functions to instances with a fixed number of variables is polynomial-time tractable.   
The restriction to a fixed number of variables has been studied for several problems in computational optimisation, and usually this kind of restriction has led  to an improvement in the computational complexity: two remarkable examples of this situation are the combinatorial polynomial-time algorithm of Megiddo \cite{Megiddo:1984:LPL:2422.322418} to solve the restriction  to a fixed number of variables of \textsc{Linear Programming} (in its full generality, Linear Programming can be solved in polynomial-time (see, e.g., \cite{Khachiyan, Karmarkar1984, Wright05theinterior-point}), but all the known algorithms rely on approximation procedures); and the algorithm of Lenstra \cite{Lenstra} to solve the restriction of Integer Programming Feasibility   (which is an NP-complete problem in its full generality) to a fixed number of variables.

\section{Preliminaries}

We adopt the following notation: $\Q$ denotes the set of rational numbers, and
 $x_i$ denotes the $i$-th component of a tuple $x$. 
 We start with some preliminaries on the cost functions that we want to take into account. 
 \begin{definition}
	A \emph{cost function  over $\Q$} is a function $f \colon \Q^{n} \to {\mathbb Q} \cup \{+\infty\}$, for a positive integer $n$.
Here, $+\infty$ is an extra element with the expected properties that for all $c \in {\mathbb Q} \cup \{+\infty\}$
\begin{align*}
(+\infty) + c & = c + (+\infty) = +\infty \\
\text{  and } c & < +\infty \text{ iff } c \in {\mathbb Q}.
\end{align*} 
A cost function $f \colon \Q^{n} \to {\mathbb Q} \cup \{+\infty\}$ can  also be seen as  a partial function such  that $f$ is not defined on $x\in \Q^n$ if, and only if, $f(x)=+\infty$. 
\end{definition}

\begin{definition}\label{def:polyhedral}
	A set $C \subseteq \Q^d$ is a  \emph{polyhedral set} if it is the intersection on finitely many (open or closed) half spaces, i.e., it can be specified by a conjunction of finitely many linear constraints, i.e., for some $r \in \mathbb N$ there exist linear functions $f_i\colon \Q^{d} \to \Q$, for $1\leq i\leq r$, such that
	\[C=\left\{x \in \Q^d \mid \bigwedge_{i=1}^p(f_i(x)\leq 0) \wedge \bigwedge_{i=p+1}^q (f_i(x)< 0) \wedge  \bigwedge_{i=q+1}^r(f_{i}(x)= 0) \right\}.\]
Observe that non-empty polyhedral sets in $\Q^d$ are, in particular, convex sets.
\end{definition}

A polyhedral set $C \subseteq \Q^d$ is \emph{open} if it is the intersection of finitely many open half spaces, i.e., for some $p\in \mathbb N$ there exist $f_i\colon \Q^{d} \to \Q$ linear functions,  $1\leq i\leq p$, such that
\[C=\{x \in \Q^d \mid \bigwedge_{i=1}^p(f_i(x)< 0)\}.\] Similarly, a polyhedral set $C \subseteq \Q^d$ is \emph{closed} if it is the intersection of finitely many closed half spaces, i.e., for some $q\in \mathbb N$ there exist $f_i\colon \Q^{d} \to \Q$ linear functions,  $1\leq i\leq p$, such that
\[C=\left\{x \in \Q^d \mid \bigwedge_{i=1}^p(f_i(x)\leq 0) \wedge \bigwedge_{i=p+1}^q (f_i(x)=0)\right\}.\]
A polyhedral set $C \subseteq \Q^n$ is \emph{bounded} if it is  bounded as a subset of $\Q^n$.
We remark that the infimum of a linear function in a closed and  bounded polyhedral set is a proper minimum; while the infimum of a linear function in an open or unbounded polyhedral set is  attained only if the linear function is constant.

\begin{definition}[\cite{rockafellar1998variational}, Definition 2.47] \label{rem:PLrepresentation}
	A function $f\colon \Q^d \rightarrow \QQ$ is a \emph{piecewise linear (PL)}  if its domain, $\dom(f)$, can be represented as the union of finitely many polyhedral sets, relative to each of which $f(x)$ is given by a linear expression, i.e., there exist finitely many mutually disjoint $C_1,\ldots, C_m$ polyhedral sets such that 
	$\bigcup_{i=1}^m C_i=\dom(f)\subseteq \Q^d$, and \[f(x_1,\ldots,x_d)=\begin{cases}
	a^i_0+a^i_1 x_1+\cdots+a^i_d x_d & \text{if } (x_1,\ldots,x_d) \in C_i\\
	+\infty   & \text{if } (x_1,\ldots,x_d) \in  \Q^d\setminus \dom(f)
	\end{cases}\]
	where $a^i=(a^i_0,a^i_1,\ldots,a^i_d) \in \Q^{d+1}$, for $1 \leq i \leq m$. 
	Piecewise linear functions are sometimes called \emph{semilinear} functions.
\end{definition} 


We are now ready to formally define the computational problem that we want to focus on.

\begin{definition}\label{vcspdef}
	Let $d$ be a positive integer, and let $V:=\{x_1,\ldots,x_d\}$ be a set of variables. An \emph{instance} $I$ of the \emph{valued constraint satisfaction problem (VCSP) for PL cost functions with variables in $V$} consists of
 an expression $\phi$ of the form
		\[\sum_{i=1}^{m} f_i(x^i_1,\ldots, x^i_{\ar(f_i)})\]
		where $f_1,\dots,f_m $ are finitely many PL cost functions  and all the $x^i_j$ are variables from $V$.
	The task is to find the  \emph{infimum cost} of $\phi$, defined as
	\[\inf_{\alpha \colon V \to \Q}\sum_{i=1}^{m} f_i(\alpha(x^i_1),\ldots, \alpha(x^i_{\ar(f_i)})),\]  and to decide whether it is attained, i.e., whether it is a proper minimum or not.
\end{definition}

The computational complexity of the $\vcsp$ for PL cost functions with variables from a fixed set $V$  depends on how the cost functions are represented in the input instances. 
We fix a representation of cost functions which is strictly related both to the mathematical properties of piecewise linear functions and to the algorithmic procedures and mathematical tools that we want to use to deal with them.

\begin{definition}{\textbf{Representation of PL cost functions.}}\label{def:reprsemilinear}
	We assume that a PL cost function is given by  a list of linear constraints, specifying the polyhedral sets, and a list of linear polynomials and $+\infty$ symbols,  defining the value of the function relatively to each polyhedral set. The linear constraints and the linear polynomials are encoded by the list of their rational coefficients, and $+\infty$ is represented by a special symbol. The constants for (numerators and denominators of) rational coefficients for linear constraints and linear polynomials are represented in binary.
\end{definition}



\textsc{Linear Programming}  is  an example of a problem which can be formulated in our setting; it is also a tool that 
plays an important role later in the paper.

\begin{definition}\textsc{Linear Programming} (\emph{LP}) is an optimisation problem with a linear objective function and a set of linear constraints imposed upon a given set of underlying variables. A linear program has the form \begin{align*}
	& \text{minimise }&\sum_{j=1}^{n}c_jx_j\\
	&\text{subject to }
	& \sum_{j=1}^{n}a_{j}^ix_j\leq b^i, \quad & \text{for } i \in \{1,\ldots,m\}. 
	\end{align*}\end{definition}
	This problem has $n$ variables $x_j$ (ranging over the rationals or over the real numbers) and $m$ linear inequalities constraints. The coefficients $c_j$, $a_{j}^i$, and $b^i$ are rational numbers, 
	for all $j \in \{1,\ldots,n\}$, and all $i \in \{1,\ldots, m\}$. The linear constraints, $\sum_{j=1}^{n}a_{j}^ix_j\leq b^i$, specify a polyhedral set, namely the \emph{feasibility polytope} over which the objective function has to be optimised. 
	
	An algorithm solving LP  either finds a point in the feasibility polytope where the objective function has the smallest value if such a point exists, or it reports that the instance is infeasible (in this case we assume that the output of the algorithm is $+\infty$), or it reports that the infimum of the objective function is $- \infty$ (in this case we assume that the output of the algorithm is $-\infty$).
	
	The \emph{\textsc{Linear Program Feasibility} problem}, \emph{LPF},  is a decision problem having the form of a standard linear program but without any objective function to minimise. 
	The output of an algorithm solving LPF is ``no" or ``yes", respectively, depending on whether the  polyhedral set defined by the linear constraints is  empty or not. 
	Both LP and LPF can be solved in polynomial time (see, e.g., \cite{Khachiyan, Karmarkar1984, Wright05theinterior-point}). 

In the remainder of the paper, we use the fact that \textsc{Linear Program Feasibility} for a set of linear constraints containing also strict inequalities can be solved in polynomial time. Given a set of linear constraints $l$ and a linear expression $\obj$, we denote by $LPF(l)$ the LPF instance defined by the linear constraints in $l$, and we denote by $LP(l,\obj)$ the LP instance defined by the linear constraints in $l$ and by the objective function $\obj$.
			
We remark that in  LP and LPF the feasibility polytope is defined by weak linear inequalities, i.e., by linear constraints of the form $\sum_{j=1}^{n}a_{j}x_j\leq b$. The feasibility of a set of linear constraints containing also strict linear inequalities (i.e., of the form $\sum_{j=1}^{n}a_{j}x_j< b$) can be solved by solving a linear number of linear programs, as shown in \cite{JonssonBaeckstroem}, where the authors give a polynomial-time algorithm deciding the feasibility of a set of \emph{Horn disjunctive linear constraints}. However, the feasibility of a set of linear constraints containing strict and weak linear inequalities can be decided by solving only one LP instance.

	\begin{lemma}[Motzkin Transposition Theorem \cite{MotzkinTransp,Roos2009}]\label{lemma:motzkintranspth}
		Let $A \in \Q^{k_1\times d}$, and $B \in \Q^{k_2\times d}$ be matrices such that $\max(k_1,d)\geq 1$. The system  \begin{equation*}\begin{cases}Ax<0\\
		Bx\leq0
		\end{cases}\end{equation*}
		has a solution $x \in \Q^d$ if, and only, if the system \begin{equation*}\begin{cases}A^{T}y+B^{T}z=0\\
		y\geq 0,\; z \geq 0
		\end{cases}\end{equation*} 
		does not admit a solution $(y,z)\in \Q^{k_1+k_2}$ such that $y \neq (0,\ldots,0)$.
	\end{lemma}
	
	\begin{proposition}
		The \textsc{Linear Program Feasibility}  problem  (LPF) for a finite set of strict or weak linear inequalities is polynomial-time many-one reducible to LP and, therefore, it can be solved in polynomial time.
	\end{proposition}	
	
	\begin{proof}
		Let us assume that the linear constraints in the input consist of $k_1$ strict inequalities, and $k_2$ weak inequalities, i.e., we have to check the satisfiability of the following system  \begin{equation}\label{eq:LPF1}\begin{cases}\sum_{i=1}^da_{j,i}x_i+a_{j,d+1}<0 &\text{ for } 1\leq j \leq k_1\\ \sum_{i=1}^db_{j,i}x_i+b_{j,d+1}\leq0 &\text{ for } 1\leq j \leq k_2,\\
		\end{cases}\end{equation}
		Let us first observe that the system (\ref{eq:LPF1}) is equivalent to the following one
		\begin{equation}\label{eq:LPF2}\begin{cases}\sum_{i=1}^{d+1}a_{j,i}t_i<0 &\text{ for } 1\leq j \leq k_1\\
		-t_{d+1}<0\\
		\sum_{i=1}^{d+1}b_{j,i}t_i\leq0 &\text{ for } 1\leq j \leq k_2.\\
		\end{cases}\end{equation}
		Indeed, if $(t_1,\ldots,t_{d},t_{d+1})$ is a solution for (\ref{eq:LPF2}), then $(x_1,\ldots,x_d)$ with $x_i:=\frac{t_i}{t_{d+1}}$ is a solution for (\ref{eq:LPF1}); vice versa if $(x_1,\ldots,x_d)$ is a solution  for  (\ref{eq:LPF1}), then $(x_1,\ldots,x_{d},1)$ is a solution for (\ref{eq:LPF2}). Let us consider the following linear program \begin{align}\nonumber
			 &\text{minimise }&\sum_{j=1}^{k_1+1}(-y_j)\\\nonumber
			&\text{subject to }
			& A^{T}y+B^{T}z 
			\\\label{eq:LPF4}
			&&		-y\leq 0\\\nonumber&&-z\leq 0,
			\end{align}
			 with variables $y_1,\ldots,y_{k_1+1},z_1,\ldots,z_{k_2}, $
		where $A\in \Q^{(k_1+1)\times(d+1)}$ is the matrix such that $(A)_{ji}=a_{ji}$ for $1\leq j \leq k_1$ and $1\leq i \leq d+1$, and such that the $(k_1+1)$-th row of $A$ is $(0, \ldots, 0, -1)$; and the matrix $B \in \Q^{(k_2)\times(d+1)}$ is such that $(B)_{ji}=b_{ji}$. 
		 
		Observe that the linear program (\ref{eq:LPF4}) can be computed in polynomial time (in the size of the input).
		By Lemma \ref{lemma:motzkintranspth}, the system  (\ref{eq:LPF2}) is satisfiable if, and only if, the  feasibility polytope determined by the linear constraints in (\ref{eq:LPF4}) does not admit a solution $(y,z)\in \Q^{(k_1+1)+k_2}$ such that $y \neq (0,\ldots,0)$.  If the output of the algorithm for LP on instance~(\ref{eq:LPF4}) is $+\infty$ or a tuple having $0$ in the first $k_1+1$ coordinates, then  the system (\ref{eq:LPF1}) is satisfiable, and therefore we accept.
		Otherwise, if the output is $-\infty$ or a tuple $(y,z)\in \Q^{(k_1+1)+k_2}$ such that $y \neq (0,\ldots,0)$, then the system (\ref{eq:LPF1}) is not satisfiable and we reject.
\end{proof}

\section[PL VCSPs with Fixed Number of Variables]{PL VCSPs with Fixed Number of Variables}\label{sec:1}
 We exhibit a polynomial-time algorithm that solves the VCSP for PL cost functions  having variables from a fixed finite  set. 
 We assume that the input  VCSP instance is given as a sum of PL cost functions represented as in Definition \ref{def:reprsemilinear}.
Our algorithm computes the infimum of the objective function, and specifies whether it is attained, i.e., whether it is a proper minimum.

The following theorem uses an idea that appeared in  \cite{Bodirsky2017ConstraintSP}, Observation 17.

\begin{theorem}
	Let $V$ be a  finite set of variables. Then  there is a polynomial-time  algorithm that solves  the $\vcsp$ for PL cost functions having variables in $V$. \end{theorem}

\begin{proof}

	We prove  that Algorithm \ref{alg:fixedarity} correctly solves the VCSP for PL cost functions with variables in $V$ in polynomial time.
	An input of an instance of the VCSP is a representation of an objective function $\phi$ as the sum  of a finite number of given cost functions, $f_1,\ldots, f_n$, applied to some of the variables in $V=\{x_1,\ldots,x_d\}$, that is, \[\phi(x_1,\ldots,x_d)=\sum_{i=1}^nf_i(x^i),\] where $x^i \in V^{\ar(f_i)}$ for $1\leq i \leq d$.  
	We can assume that the cost function $f_i$ is defined  for every $x \in \Q^d$  by \[f_i(x)=\begin{cases}
	\sum_{j=1}^{d}a_{j}^{i,l} x_{j}+b^{i,l} &  \text{if } C_{i,l}(x), \text{ for some } 1\leq l \leq m_i\\
	+\infty & \text{otherwise.} 
	\end{cases}\] 
	For every $1\leq l \leq m_i$ the formulas $C_{i,l}(x)$ have the following form:  	\[C_{i,l}(x)= \bigwedge_{j=1}^p(h^{i,l}_j(x)\leq 0) \wedge \bigwedge_{j=p+1}^q (h^{i,l}_j(x)< 0) \wedge  \bigwedge_{j=q+1}^r(h^{i,l}_j(x)= 0),\] for some $p, q, r \in \mathbb N$ and  for some linear polynomials $h^{i,l}_j\colon \Q^d \to \Q$, where $1\leq j \leq r$.
We assume that the cost functions $f_i$ are represented as in Definition~\ref{def:reprsemilinear}.
	
	Algorithm~\ref{alg:fixedarity} first extracts the list of linear polynomials $p_1,\ldots,p_k$ that appear in the finite set of linear constraints defining some cost function $f_i$ 
	, i.e.,
	\[\{p_1,\ldots,p_k\}:=\bigcup_{i=1}^n \bigcup_{l=1}^{m_i}\bigcup_{j} \left\{h^{i,l}_j\right\}.\] 
	Observe that the linear polynomials $p_1,\ldots, p_k$ decompose the space $\Q^d$ into $\sigma$ polyhedral sets, where \begin{equation}\label{formula}\sigma\leq \tau_d(k)=\sum_{i=0}^{d}2^i\binom{k}{i}\end{equation} and that this bound is tight, i.e., $\sigma=\tau_d(k)$, whenever the hyperplanes defined by $p_i(x)=0$, for $1\leq i \leq k$, are in general position.
	
	Inequality~(\ref{formula}) can be verified by induction on the number of hyperplanes, $k$. Clearly, for all $d \in \mathbb N$, one hyperplane divides $\Q^d$ into $3=2^0+2^1$ polyhedral sets. Suppose now that $k \geq 3$ and that Inequality~(\ref{formula}) is true for every $d$ and for at most $k-1$ hyperplanes. Suppose that the $k$ hyperplanes are in general position (we get in this way the upper bound $\tau_d(k)$).
	Observe that, by adding the hyperplanes one-by-one, the $k$-th hyperplane intersects at most $\tau_{d-1}(k-1)$ of the polyhedral sets obtained until the previous step. In fact, this number is equal to the number of polyhedral sets in which a hyperplane, that is a subspace of dimension $d-1$, is divided by $k-1$ subspaces of dimension $d-2$. 
	
	Suppose that we know how the space is decomposed into polyhedral sets by the hyperplanes $p_1(x)=0, \ldots, p_{k-1}(x)=0$. Adding $p_k(x)=0$ to the list of hyperplanes decomposing the space, each one of the polyhedral sets intersecting it is divided in three polyhedral sets (corresponding to $p_k(x)<0$, $p_k(x)=0$, and $p_k(x)>0$, respectively). Summing up, at every step we add to the ``old polyhedral sets"  (i.e., polyhedral sets obtained until the previous step) two more polyhedral sets for each of the old ones intersecting $p_k(x)=0$, then it follows that \[\tau_d(k)=\tau_d(k-1)+2\tau_{d-1}(k-1).\] Using this equality and the inductive hypothesis we obtain \begin{align*}\tau_d(k)&=2\sum_{i=0}^{d-1}2^i\binom{k-1}{i}+\sum_{i=0}^{d}2^i\binom{k-1}{i}=\sum_{i=1}^{d}2^i\bigg(\binom{k-1}{i-1}+\binom{k-1}{i}\bigg)+1\\&=\sum_{i=1}^{d}2^i\binom{k}{i}+1=\sum_{i=0}^{d}2^i\binom{k}{i}.\end{align*}
	In particular, the number  $\sigma$ of polyhedral sets is bounded by a linear polynomial in $k$, and the Algorithm~\ref{alg:fixedarity} produces a tree, that a priori has $3^k$ branches, but that actually has $O(k)$ branches. 
	
	The algorithm computes the list of all non-empty polyhedral sets by computing at most $\sum_{i=1}^{k-1}\tau_d(i)$ instances of linear program  feasibility, and then it computes the infimum of the objective function in every non-empty polyhedral set by computing at most $3\tau_d(k)$ linear programs. Observe that the only closed and bounded non-empty polyhedral sets computed by Algorithm~\ref{alg:fixedarity} are $0$-dimensional subspaces, i.e., points,  and that all the other polyhedral sets computed are open or unbounded.  Therefore, in order to check whether the infimum in a polyhedral set $C$ is a minimum, it is enough to check whether  the objective function in $C$ is constant, that is whether the infimum in $C$  is equal to the supremum in  $C$. This is done by our algorithm solving at most $3\tau_d(k)$ further linear programs. The  linear expression of the objective function in a polyhedral set can be computed by running a  number of  \textsc{Linear Program Feasibility} instances that is polynomial in the size of the input instance.
	Globally, the running time of Algorithm~\ref{alg:fixedarity} is polynomial in the size of the input. 
 
\end{proof}

\noindent\begin{minipage}{\textwidth}

\begin{algorithm}[H] \label{alg:fixedarity}
	
	\SetAlgoNoLine
	\KwIn{$\phi(x)=f_1(x)+\cdots+f_n(x)$ with $f_i(x)=f_{ij}(x)$ if $x \in C_{ij}$, and the $C_{ij}$'s  
		each given as a finite set of linear conditions, for   $1\leq j \leq n_i$, and $1\leq i \leq n$.}
	\KwOut{$(\val,\attr)$ where $\val$ is the value of the infimum of the objective function, and $\attr$ is a string which specifies whether $\val$ is attained ($\attr=\text{``}\min\text{"}$) or not ($\attr=\text{``}\inf\text{"}$). }
	$\{p_1,\ldots,p_k\}:=$  the set of all the linear functions appearing in the $C_{ij}$'s\;
	$L:=\{\{\}\}$ (the set of polyhedral sets in which the $p_i$'s divide the space)\;
	
	\For{$ i=1,...,k$
	}{
		\For{each $l$ in $L$
		}{
			$l_{-1}:=l \cup\{p_i<0\}$\;
			$l_{0}:=l\cup\{p_i=0\}$\;
			$l_{1}:=l \cup\{-p_i<0\}$\;
			\For{$j=-1,0,1$}{
			\If{LPF($l_{j})=$ yes
			}{
				$L:=(L \setminus\{l\})\cup \{l_j\}$}
		}}
	}  
	$\val:=+\infty$\;     
	$\attr:''\inf"$\;
	\For{each $l$ in $L$
	}{
	
	$l_c:=\{\}$ (the closure of $l$)\;
	\For{each $c\in l$}{\eIf{$c$ is of the form $(p<0)$}{$l_c:=l_c\cup\{p\leq 0\}$} {$l_c:=l_c \cup\{c\}$} 
	} 
	\For{$i=1,\ldots,n$ 
		}{
		$g_i:=+\infty$\;
		\For{$j=1,\ldots, n_i$ }{
			\If{LPF$(l \cup C_{ij})=$ yes
			}{
				$g_i(x):=f_{ij}(x)$
			}}
		}
		$\obj:=\sum_{i=1}^{n}g_i(x)$\;
		$m:= $LP$(l_c$, $\obj) $ (the infimum of $\phi$ in  $l$)\;
		$M:= - $LP$(l_c$, $-\obj )$ (the supremum of $\phi$ in  $l$)\;
			\If{$m<\val$}{\eIf{$m=M$}{$\attr:=\text{``}\min\text{"}$(the infimum is attained iff $\phi$ is constant in $l$)}{$\attr:=\text{``}\inf\text{"}$}} 
	}
\Return{$(\val,\attr)$}\;	
	\caption{Algorithm for PL VCSPs with a fixed number of variables}

\end{algorithm}	
\end{minipage}


\section{Conclusion and Future Work}
We have provided a polynomial-time algorithm solving the VCSP for piecewise linear cost functions  having a fixed number of variables.
 In the future, we would like to continue this line of research by studying the computational complexity of the VCSP with a fixed number of variables and \emph{semialgebraic} cost functions. A function $f \colon \R^n\to \RR$ is called semialgebraic if its domain can be represented as the union of finitely many  \emph{basic semialgebraic sets} (see \cite{Bodirsky2017ConstraintSP}) of the form $\{x \in \R^n \mid \chi(x)\}$ where $\chi$ is a conjunction of (weak or strict)
polynomial inequalities with integer coefficients, relative to each of which $f(x)$ is given by a polynomial expression with integer coefficients. 

The VCSP for all semialgebraic cost function is  equivalent to the existential theory of the reals (see~\cite{Bodirsky2017ConstraintSP}),  which  is in \textsc{PSpace} (see~\cite{Canny:1988:AGC:62212.62257}). The restriction of the feasibility problem associated with a semialgebraic VCSP to a fixed number of variables  is polynomial-time tractable by  cylindrical decomposition (cf.\;\cite{CADcollins}). However,  we do not know whether with this approach can solve our optimisation problem in polynomial time.
Another contribution related with our open problem was given in~\cite{KhachiyanPorkolab2000} by Khachiyan and Porkolab, who  proved that the problem of minimising a convex polynomial objective function with integer coefficients over a fixed number of integer  variables, subject to polynomial constraints with integer coefficients that define a convex region, can be solved in polynomial time in the  size of the input.

\bibliographystyle{abbrv}
\bibliography{local}

\def\cprime{$'$} \def\cprime{$'$}
\begin{thebibliography}{10}

\bibitem{BodirskyGrohe}
M.~Bodirsky and M.~Grohe.
\newblock Non-dichotomies in constraint satisfaction complexity.
\newblock In L.~Aceto, I.~Damgard, L.~A. Goldberg, M.~M. Halld\'orsson,
  A.~Ing\'olfsd\'ottir, and I.~Walukiewicz, editors, {\em Proceedings of the
  International Colloquium on Automata, Languages and Programming, {ICALP}'08},
  Lecture Notes in Computer Science, pages 184 --196. Springer Verlag, July
  2008.

\bibitem{Bodirsky2017ConstraintSP}
M.~Bodirsky and M.~Mamino.
\newblock Constraint satisfaction problems over numeric domains.
\newblock In {\em Dagstuhl Follow-Ups}, volume~7. Schloss
  Dagstuhl-Leibniz-Zentrum fuer Informatik, 2017.

\bibitem{BulatovFVConjecture}
A.~A. Bulatov.
\newblock A dichotomy theorem for nonuniform {CSP}s.
\newblock In {\em 58th {IEEE} Annual Symposium on Foundations of Computer
  Science, {FOCS}'17 , Berkeley, CA, USA, October 15-17, 2017}, pages 319--330,
  2017.

\bibitem{Canny:1988:AGC:62212.62257}
J.~Canny.
\newblock Some algebraic and geometric computations in {PSPACE}.
\newblock In {\em Proceedings of the Twentieth Annual ACM Symposium on Theory
  of Computing, {STOC}'98}, pages 460--467, New York, NY, USA, 1988. ACM.

\bibitem{CADcollins}
G.~E. Collins.
\newblock Quantifier elimination for real closed fields by cylindrical
  algebraic decompostion.
\newblock In H.~Brakhage, editor, {\em Automata Theory and Formal Languages},
  pages 134--183, Berlin, Heidelberg, 1975. Springer Berlin Heidelberg.

\bibitem{GareyJohnson}
M.~Garey and D.~Johnson.
\newblock {\em A guide to {NP}-completeness}.
\newblock CSLI Press, Stanford, 1978.

\bibitem{JonssonBaeckstroem}
P.~Jonsson and C.~B\"ackstr\"om.
\newblock A unifying approach to temporal constraint reasoning.
\newblock {\em Artificial Intelligence}, 102(1):143--155, 1998.

\bibitem{Karmarkar1984}
N.~Karmarkar.
\newblock A new polynomial-time algorithm for linear programming.
\newblock {\em Combinatorica}, 4(4):373--395, 1984.

\bibitem{Karp1972}
R.~M. Karp.
\newblock Reducibility among combinatorial problems.
\newblock In R.~E. Miller, J.~W. Thatcher, and J.~D. Bohlinger, editors, {\em
  Complexity of Computer Computations: Proceedings of a symposium on the
  Complexity of Computer Computations}, pages 85--103, Boston, MA, 1972.
  Springer US.

\bibitem{Khachiyan}
L.~Khachiyan.
\newblock A polynomial algorithm in linear programming.
\newblock {\em Doklady Akademii Nauk SSSR}, 244:1093--1097, 1979.

\bibitem{KhachiyanPorkolab2000}
L.~Khachiyan and L.~Porkolab.
\newblock Integer optimization on convex semialgebraic sets.
\newblock {\em Discrete $\&$ Computational Geometry}, 23:207--224, 02 2000.

\bibitem{GenVCSP15}
V.~Kolmogorov, A.~A. Krokhin, and M.~Rolinek.
\newblock The complexity of general-valued {CSP}s.
\newblock In {\em {IEEE} 56th Annual Symposium on Foundations of Computer
  Science, {FOCS}'15, Berkeley, CA, USA, 17-20 October, 2015}, pages
  1246--1258, 2015.

\bibitem{KolmogorovThapperZivny}
V.~Kolmogorov, J.~Thapper, and S.~\v{Z}ivn\'{y}.
\newblock The power of linear programming for general-valued {CSP}s.
\newblock {\em {SIAM} J. Comput.}, 44(1):1--36, 2015.

\bibitem{KozikOchremiak15}
M.~Kozik and J.~Ochremiak.
\newblock Algebraic properties of valued constraint satisfaction problem.
\newblock In {\em Automata, Languages, and Programming - 42nd International
  Colloquium, {ICALP}'15, Kyoto, Japan, July 6-10, 2015, Proceedings, Part
  {I}}, pages 846--858, 2015.

\bibitem{Lenstra}
H.~W. Lenstra.
\newblock Integer programming with a fixed number of variables.
\newblock {\em Mathematics of Operations Research}, 8(4):538--548, 1983.

\bibitem{Megiddo:1984:LPL:2422.322418}
N.~Megiddo.
\newblock Linear programming in linear time when the dimension is fixed.
\newblock {\em J. ACM}, 31(1):114--127, Jan. 1984.

\bibitem{MotzkinTransp}
T.~S. Motzkin.
\newblock {\em Beitr\"{a}ge zur {T}heorie der {L}inearen {U}ngleichungen}.
\newblock PhD thesis, University of Basel, Azriel, Jerusalem, 1936.
\newblock English Translation: Contributions to the theory of linear
  inequalities, RAND Corporation Translation 22, The RAND Corporation, Santa
  Monica, California, 1952. Reprinted in: Theodore S. Motzkin: Selected Papers
  (D. Cantor, B. Gordon, B. Rothschild, eds.), Birkh\"{a}user, Boston,
  Massachussetts, 1983, pp. 1-80.

\bibitem{rockafellar1998variational}
R.~T. Rockafellar and R.~J.~B. Wets.
\newblock {\em Variational Analysis}, volume 317.
\newblock Springer-Verlag, Berlin, 1998.

\bibitem{Roos2009}
K.~Roos.
\newblock Linear optimization: Theorems of the alternative.
\newblock In C.~A. Floudas and P.~M. Pardalos, editors, {\em Encyclopedia of
  Optimization}, pages 1878--1881. Springer US, Boston, MA, 2009.

\bibitem{Wright05theinterior-point}
M.~H. Wright.
\newblock The interior-point revolution in optimization: history, recent
  developments, and lasting consequences.
\newblock {\em Bull. Amer. Math. Soc.}, 42:39--56, 2005.

\bibitem{ZhukFVConjecture}
D.~Zhuk.
\newblock A proof of {CSP} dichotomy conjecture.
\newblock In {\em 58th {IEEE} Annual Symposium on Foundations of Computer
  Science, {FOCS}'17 , Berkeley, CA, USA, October 15-17, 2017}, pages 331--342,
  2017.

\end{thebibliography}

\end{document}